\documentclass{article}
\usepackage{amsmath}
\usepackage{amsfonts}
\usepackage{amsthm}
\usepackage{latexsym}
\usepackage{amssymb}
\usepackage{verbatim}
\usepackage{enumerate}
\usepackage[inner]{showlabels}

\def\noteson{%
\gdef\note##1{\mbox{}\marginpar[##1]{%
     ##1}}}

\noteson

\font\logic=msam10 at 10pt
\newcommand{\forces}{\mbox{\logic\char'015}}
\newcommand{\restrict}{\mbox{\logic\char'026}}
\def\underTilde#1{{\baselineskip=0pt\vtop{\hbox{$#1$}\hbox{$\sim$}}}{}}

\newcommand{\uTDelta}{\underTilde{\Delta}}
\newcommand{\uTSigma}{\underTilde{\Sigma}}
\newcommand{\uTPi}{\underTilde{\Pi}}

\newcommand{\less}{\mathord{<}}

\newcommand{\breals}{\omega^{\omega}}

\newcommand{\ot}{\mathrm{o.t.}}

\newcommand{\dom}{\operatorname{dom}}
\newcommand{\creals}{{^{\omega}}2}
\newcommand{\bbP}{\mathbb{P}}
\newcommand{\proj}{\mathrm{proj}}
\newcommand{\bq}{{\bf q}}
\newcommand{\bQ}{{\bf Q}}

\newcommand{\bL}{\mathbf{L}}
\newcommand{\bV}{\mathbf{V}}

\newcommand{\ZFC}{\mathsf{ZFC}}
\newcommand{\ZF}{\mathsf{ZF}}

\newcommand{\cA}{\mathcal{A}}

\newcommand{\cP}{\mathcal{P}}
\newcommand{\cT}{\mathcal{T}}

\newtheorem{thrm}{Theorem}[section]
\newtheorem{lem}[thrm]{Lemma}

\newtheorem{prop}[thrm]{Proposition}

\newtheoremstyle{hdefinition}%
  {\topsep}%
  {\topsep}%
  {\upshape}
  {}%
  {\bfseries}%
  {.}
  { }%
  {\thmnumber{#2 }\thmname{#1}\thmnote{ \rm(#3)}}%

\newtheoremstyle{hclaim}%
  {\topsep}%
  {\topsep}%
  {\itshape}%
  {}%
  {\bfseries}%
  {.}
  { }%
  {\thmname{#1}\thmnote{ \rm#3}}%

\newtheoremstyle{hnotation}%
  {\topsep}%
  {\topsep}%
  {\upshape}%
  {}%
  {\bfseries}%
  {.}
  { }%
  {\thmname{#1}\thmnote{ \rm#3}}%

\theoremstyle{hclaim}
\newtheorem*{claim*}{Claim}

\theoremstyle{hdefinition}
\newtheorem{df}[thrm]{Definition}

\newtheorem{remark}[thrm]{Remark}

\theoremstyle{hclaim}

\theoremstyle{hnotation}

\begin{document}

\title{Universally measurable sets may all be $\uTDelta^{1}_{2}$}


\author{Paul B. Larson\thanks{Supported in part by NSF Grants
  DMS-1201494 and DMS-1764320.}\\
 Miami University\\
Oxford, Ohio USA\\
\and Saharon Shelah\thanks{Research partially supported by NSF grant no: 1101597, and by the
European Research Council grant 338821. Paper No. 1178 on Shelah's list. }\\
Hebrew University of Jerusalem\\Rutgers University}




\maketitle

\begin{abstract}
  We produce a forcing extension of the constructible universe $\bL$ in which
  every universally measurable set of reals is $\uTDelta^{1}_{2}$, partially answering question CG from David Fremlin's problem list \cite{FQL}. The analogous result for category holds in the same model.
\end{abstract}

We let $\omega^{\omega}$ denote the Baire space, the set of functions from $\omega$ to $\omega$, and refer to its elements as \emph{reals}.
A subset of a Polish space $X$ is said to be \emph{universally measurable} if it is measured by the completion of any $\sigma$-additive Borel measure on
$X$. Equivalently, $A \subseteq X$ is universally measurable if and only if $f^{-1}[A]$ is Lebesgue measurable whenever $f \colon \breals \to X$ is a Borel function (see \cite{Ke, Ni}, and 434D of \cite{F4}, for instance). This characterization induces the corresponding notion for category : we will say that a set $A \subseteq X$ is \emph{universally categorical} if and only if $f^{-1}[A]$ has the property of Baire whenever $f \colon \breals \to X$ is a Borel function. (The term \emph{universally Baire} has already been established with a different meaning \cite{FMW}, implying both universal measurability and  universal categoricity.) The collections of the universally measurable subsets of $X$ and the universally categorical subsets of $X$ are both $\sigma$-algebras on $X$.


A subset $A$ of a Polish space $X$ is $\uTDelta^{1}_{2}$ if $A$ and $X \setminus A$ are continuous images of coanalytic sets. We refer the reader to \cite{Ke} for background on this definition and for information on the projective sets in general. In this paper we show how to force over any model of the form $\bL[a]$, where $a \subseteq \omega$, to produce a model in which all subsets of $\breals$ which are either universally measurable or universally categorical are $\uTDelta^{1}_{2}$. Since all uncountable Polish spaces are Borel-isomorphic (see Theorem 17.41 of \cite{Ke}), this implies the same fact for subsets of any Polish space. The following then is the main theorem of this paper (Theorem \ref{mainthrm} gives a more explicit statement).

\begin{thrm}\label{specthrm} If, for some $a \subseteq \omega$, $\bV \mathord{=} \bL[a]$, then there is a proper forcing extension in which every universally measurable subset of any Polish space is $\uTDelta^{1}_{2}$, and every universally categorical subset of any uncountable Polish space is $\uTDelta^{1}_{2}$.
\end{thrm}

The result for universally measurable sets answers part of problem CG on David Fremlin's problem list \cite{FQL}.






 Since there are only continuum many $\uTDelta^{1}_{2}$ sets, our result also strengthens (modulo the anti-large cardinal hypothesis $\bV \mathord{=} \bL[a]$) a previous result of the authors with Itay Neeman \cite{LrNeSh}, which showed the consistency of the statement that the set of universally measurable sets has the same cardinality as $\breals$. We note that (unlike the results in \cite{LrNeSh}) some anti-large cardinal hypothesis
is needed for the results in this paper, since the existence of infinitely many Woodin cardinals for instance implies that every projective set of reals is universally measurable \cite{ShWoo}, and there are (assuming $\ZF$) projective sets which are not $\uTDelta^{1}_{2}$ (see Theorem 37.7 of \cite{Ke}).


\section{Outline of the proof}\label{outlinesec}

The proof of Theorem \ref{specthrm} is an application of forcing machinery developed by the second author and his collaborators (especially \cite{Sh64, Sh98}, but we also make use of results from \cite{KellSh}). The proof proceeds by forcing over a model of the form $\bL[a]$ (for any $a \subseteq \omega)$ with a countable support iteration of proper partial orders, and makes use of the following theorem, which is Theorem III.4.1 in \cite{ShPIF}.

\begin{thrm}\label{shthrm}
  Suppose that $\kappa$ is a regular cardinal such that $\mu^{\aleph_{0}} < \kappa$ for all $\mu < \kappa$ and that  $\bar{\bbP} = \langle \bbP_{\alpha}, \dot{Q}_{\beta} : \alpha \leq \kappa, \beta < \kappa \rangle$ is a countable support iteration such that each $\bbP_{\alpha}$ forces the corresponding $\dot{Q}_{\alpha}$ to be a proper forcing of cardinality less than $\kappa$. Then $\bbP_{\kappa}$ is $\kappa$-c.c., and for each $\alpha < \kappa$, $\bbP_{\alpha}$ has a dense subset of cardinality less than $\kappa$. Furthermore, for all $\alpha < \kappa$, $\bbP_{\alpha}$ forces that $2^{\aleph_{0}} < \kappa$.
\end{thrm}

We will apply this theorem with ground models of the form $\bL[a]$ for $a \subseteq \omega$, which satisfy the Generalized Continuum Hypothesis. Although we could let $\kappa$ be any regular cardinal with $\kappa^{<\kappa} =  \kappa$ we will restrict ourselves to the case $\kappa = \omega_{2}^{\bV}$. Each step of our iterations will be an $(\omega, \infty)$-distributive partial order of cardinality continuum (see Definition \ref{podef}), and will force the Continuum Hypothesis (CH) to hold. To see that Theorem \ref{shthrm} applies, we need to know that each $\bbP_{\alpha}$ preserves the statement that $2^{\aleph_{0}} < \kappa$. This is not hard to show directly for the partial orders we consider, but it also follows by applying Theorem \ref{shthrm} to the modified iteration where each $\dot{Q}_{\alpha}$ is the $\bbP_{\alpha}$-name for either our original $\dot{Q}_{\alpha}$ if CH holds, and the trivial forcing if it fails. The theorem then implies that the second case never holds. We then have from Theorem \ref{shthrm} that each $\bbP_{\alpha}$ (in our original, intended iteration) will have a dense subset of cardinality less than $\kappa$, and will therefore preserve the statement $2^{\aleph_{1}} \leq \kappa$.


We will use the following definition.


\begin{df} Given a partial order $P$ and a $\sigma$-algebra $\cA$ on $\breals$ we say that $P$ is $\cA$-\emph{representing}
if for each $A \in \cA$ (in $\bV$), every element of $\breals$ in any forcing extension by $P$ is in a Borel set with a code in $\bV$ which is either contained in or disjoint from $A$.
\end{df}

Each tail of the iterations we consider here is $\cA$-representing when $\cA$ is either the algebra of universally measurable or universally categorical sets. The overall strategy of our proof is then summarized by the following proposition.

\begin{prop}
  Suppose that $\bar{\bbP}$ is a forcing iteration of length $\kappa$ satisfying the conditions in Theorem \ref{shthrm}. Let $G \subseteq \bbP_{\kappa}$ be a $\bV$-generic filter, and for each $\alpha < \kappa$ let $G_{\alpha}$ be the restriction of $G$ to $\bbP_{\alpha}$. Let $\cA$ $(\alpha \leq \kappa)$ be such that either each $\cA_{\alpha}$ is the set of universally measurable subsets of $\breals$ in $\bV[G_{\alpha}]$ or
  each $\cA_{\alpha}$ is the set of universally categorical subsets of $\breals$ in $\bV[G_{\alpha}]$. Suppose that each tail $\bbP_{[\alpha, \kappa)}$ of $\bar{\bbP}$ is $\cA_{\alpha}$-representing. Then for each $A \in \cA_{\kappa}$ there exists an $\alpha < \kappa$ such that $A$ is equal to the union of all the Borel subsets of $A$ in $\bV[G]$ having codes in $\bV[G_{\alpha}]$.
\end{prop}

\begin{proof} Fixing $A$, we have that $A$ is the realization of a $\bbP_{\kappa}$-name $\dot{A}$ in the ground model $\bV$. We have the following facts:
\begin{itemize}
\item each $\bbP_{\alpha}$ preserves the statement $2^{\aleph_{0}} < \kappa$;
\item $\bbP_{\kappa}$ preserves the regularity of $\kappa$;
\item for each $A \subseteq \breals$ the statement that $A$ is universally measurable and the statement that $A$ is universally categorical are both $\uTPi^{1}_{2}$ in $A$.
\end{itemize}
Together these imply that for club many $\alpha < \kappa$ (more importantly, at least one), the following hold.
  \begin{itemize}
    \item For each $\bbP_{\alpha}$-name $\sigma$ for an element of $\breals$, some element of $G_{\alpha}$ decides the statement $\sigma \in \dot{A}$ (with respect to $\bbP_{\kappa}$, with $\sigma$ interpreted as a $\bbP_{\kappa}$-name).
    \item For each Borel set $B$ in $V[G]$ with a code in $\bV[G_{\alpha}]$, if $B \cap \dot{A}_{G}$ is nonempty, then so is $B \cap \dot{A}_{G} \cap \bV[G_{\alpha}]$.
    \item For each Borel set $B$ in $V[G]$ with a code in $\bV[G_{\alpha}]$, if $B \setminus \dot{A}_{G}$ is nonempty, then so is $(B \setminus \dot{A}_{G}) \cap \bV[G_{\alpha}]$.
    \item If $\dot{A}_{G}$ is universally measurable in $\bV[G]$, then $\dot{A}_{G} \cap \bV[G_{\alpha}]$ is universally measurable in $\bV[G_{\alpha}]$.
    \item If $\dot{A}_{G}$ is universally categorical in $\bV[G]$, then $\dot{A}_{G} \cap \bV[G_{\alpha}]$ is universally categorical in $\bV[G_{\alpha}]$.
  \end{itemize}
To finish the proof, let $x$ be an element of $A$. Since $A \cap \bV[G_{\alpha}]$ is in $\cA_{\alpha}$, and the tail $\bbP_{[\alpha, \kappa)}$ of $\bar{\bbP}$ is $\cA_{\alpha}$-representing, $x$ is a member of a Borel set $B$ with a code in $\bV[G_{\alpha}]$ such that $B \cap \bV[G_{\alpha}]$ is either contained in or disjoint from $A$. Since $B$ is not disjoint from $A$, $B \cap \bV[G_{\alpha}]$ is not either. So $B \cap \bV[G_{\alpha}]$ is contained in $A$, and $B$ is also.
\end{proof}


The forcing construction in this paper produces a model in which every set of reals of cardinality $\aleph_{1}$ is $\uTSigma^{1}_{2}$ (the fact that each $\bbP_{\alpha}$ preserves the inequality $2^{\aleph_{1}} \leq \kappa$ makes this possible with an iteration of length $\kappa$). Given this, any $A \subseteq \breals$ with the property that $A$ and $\breals \setminus A$ are both unions of $\aleph_{1}$-many Borel sets is $\uTDelta^{1}_{2}$ (see Remark \ref{aleph1Borelrem}).



\begin{remark}\label{reducerem}
The outline in this section reduces the proof of
 the main theorem to establishing the following regarding the iterations $\langle \bbP_{\alpha}, \dot{Q}_{\beta}: \alpha \leq \kappa, \beta < \kappa \rangle$ considered in this paper (iterations as in Definition \ref{QCFdef} and their tails):
\begin{itemize}
\item each $\bbP_{\alpha}$ forces that $\dot{Q}_{\alpha}$ is a proper $(\omega, \infty)$-distributive forcing of cardinality at most $2^{\aleph_{0}}$ making CH hold (established in Remark \ref{popropsrem} and Lemma \ref{properlem});
\item $\bbP_{\kappa}$ forces that every subset of $\breals$ of cardinality $\aleph_{1}$ is $\uTSigma^{1}_{2}$ (shown in Lemma \ref{codingitlem});
\item each tail of the iteration $\bbP_{\kappa}$ is $\cA$-representing when $\cA$ is either the set of universally measurable subsets of $\breals$ or the set of universally categorical subsets of $\breals$. This is shown in Lemma \ref{representinglem}.
\end{itemize}
\end{remark}

At the end of the paper we prove one additional result not directly related to the main theorem. The paper \cite{LrNeSh} introduced the following notion : given a ground model set $A \subseteq \breals$, the \emph{Borel reinterpretation} of $A$ in a forcing extension is the union of all the ground model Borel sets contained in $A$, each reinterpreted in the extension. This offered a characterization of the universally measurable sets as the sets $A \subseteq \breals$ with the property that the Borel reinterpretations of $A$ and $\breals \setminus A$ are complements in any extension by random forcing (in particular, random forcing is $\cA$-representing, when $\cA$ is the collection of universally measurable sets). Similarly, the universally categorical sets are the
sets $A \subseteq \breals$ with the property that the Borel reinterpretations of $A$ and $\breals \setminus A$ are complements in any extension by Cohen forcing.

Given a definable $\sigma$-ideal $\cA$ on $\breals$, we say that a partial order $P$ has the $\cA$-\emph{reinterpretation property} if it is $\cA$-representing, and, in addition the Borel reinterpretations of members of $\cA$ in the ground model are in $\cA$ as defined in forcing extensions by $P$. Many forcings have this property for the universal measurable sets, including random forcing and Sacks forcing; see \cite{LZpime}. Theorem \ref{umeasreint} shows that the iterations considered in this paper have the $\cA$-reinterpretation property when $\cA$ is the set of universally measurable subsets of $\breals$.
This is in some sense a negative result : an iteration forcing the statement ``every universally measurable set has the property of Baire" (whose consistency is still an open question) cannot have the reinterpretation property for universally measurable sets if, for instance, it is applied to a model of Martin's Axiom (more generally, to a model with a medial limit).







\section{Coding subsets of $\omega_{1}$ by reals}\label{stepsec}

We let $\Omega$ denote the set of countable limit ordinals. A \emph{ladder system} on $\omega_{1}$ is a sequence
$\langle C_{\eta} : \eta \in \Omega \rangle$ such that each $C_{\eta}$ is a cofinal subset of $\eta$ of ordertype $\omega$.
For $C$ an infinite set of ordinals and $n \in \omega$, we write $C(n)$ for the unique $\alpha \in C$ such that $|C \cap \alpha| = n$.

Given two sequences $s, t$, we write $s \trianglelefteq t$ mean that $s$ is an initial segment of $t$.
Given $s \in 2^{\less\omega}$, we let $[s]$ denote $\{ x \in \creals : s \trianglelefteq x\}$.

\begin{df}\label{denfundef}
We define a \emph{dense function} to be a partial function
$F \colon 2^{\omega} \to 2$ such that for each $s \in 2^{<\omega}$,
$F[[s] \cap \dom(F)] = 2$.
\end{df}


\begin{df}\label{podef}
We define the forcing $Q_{\bar{C}, F, g}$, where
\begin{itemize}
\item $\bar{C} = \langle C_{\eta} : \eta \in \Omega\rangle$ is a ladder system on $\omega_{1}$;
\item $F \colon 2^{\omega} \to 2$ is a dense partial function;
\item $g$ is a function from $\Omega$ to $2$.
\end{itemize}
The conditions of $Q_{\bar{C}, F, g}$ are the functions $p$ such that,
\begin{itemize}
\item the domain of $p$ is a countable ordinal $\delta_{p}$;
\item the range of $p$ is a subset of $2$;
\item for all $\eta \in (\delta_{p} + 1) \cap \Omega$,
$\langle p(C_{\eta}(i)) : i < \omega \rangle \in \dom(F)$ and
\[F(\langle p(C_{\eta}(i)) : i < \omega \rangle) = g(\eta).\]
\end{itemize}
The order on $Q_{\bar{C}, F, g}$ is extension.
\end{df}

\begin{remark}\label{popropsrem}Each partial order of the form $Q_{\bar{C}, F, g}$ has cardinality $2^{\aleph_{0}}$, as does its transitive closure (so $Q_{\bar{C}, F, g}$
is in $H(\mathfrak{c}^{+})$).
Each such partial order also forces $2^{\aleph_{0}} = \aleph_{1}$. To see this, note first that by Lemma \ref{properlem} below, $Q_{\bar{C}, F, g}$ adds no new elements of
$\breals$. To see that forcing with $Q_{\bar{C}, F, g}$ wellorders $(2^{\omega})^{\bV}$ in ordertype $\omega_{1}$, let $\bar{D}$ be a ladder system
on $\omega_{1}$ such that
\begin{itemize}
\item $D_{\eta} \cap C_{\eta} = \emptyset$ for each $\eta \in \Omega$ of ordertype greater than $\omega$ and
\item $\sup\{ D_{\eta}(0) : \eta \in \Omega \} = \omega_{1}$.
\end{itemize}
Let $G \colon \omega_{1} \to 2$ be a $V$-generic function for $Q_{\bar{C}, F, g}$, and note that each element of $(2^{\omega})^{V}$ is equal to
$\langle G(D_{\eta}(i)) : i < \omega \rangle$ for some $\eta \in \Omega$ (this follows from a standard genericity argument, and we leave the details to the reader).
\end{remark}

The following lemma establishes basic properties of the partial orders $Q_{\bar{C}, F, g}$.

\begin{lem}\label{properlem} If $\bar{C}$ is a ladder system on $\omega_{1}$, $F \colon 2^{\omega} \to 2$ is a partial dense function and
$g$ is a function from $\Omega$ to $2$ then the following hold.
\begin{enumerate}
\item For each condition $p \in Q_{\bar{C}, F, g}$, each $\gamma < \omega_{1}$ and each finite partial function $s$ from $\gamma\setminus \delta_{p}$ to $2$, there is a condition $q \leq p$ with $\delta_{q} \geq \gamma$ and $q \restrict \dom(s) = s$.
\item The partial order $Q_{\bar{C}, F, g}$ is proper and $(\omega, \infty)$-distributive.
\end{enumerate}
\end{lem}

\begin{proof}
We prove the first part by induction on $\gamma$. There is nothing to do in the case where $\gamma \leq \delta_{p}$. The case where $\gamma$ is the successor of some ordinal $\eta$ follows from the induction hypothesis for $\eta$ along with an extension to add $\eta$ to the domain of the condition $q$, agreeing with $s$ if necessary. For the case where $\gamma$ is a limit ordinal, an application of the induction hypothesis can be used to find an extension $p'$ of $p$ containing $s$. Since $F$ is dense there is a $y \in 2^{\omega}$ extending
\[\langle p'(C_{\gamma}(i)) : i < \omega,\, C_{\gamma}(i) < \delta_{p'}\rangle\] with $F(y) = g(\gamma)$. Applying the induction hypothesis $\omega$ any times (with the values $C_{\gamma}(i) + 1$ for which $C_{\gamma}(i) \not\in \dom(p')$ successively in the role of $\gamma$, and the restriction of $y$ to the values $C_{\gamma}(i)$ in the role of $s$) gives a condition $q$ as desired.

The argument for the second part is similar. Let $p_{0}$ be a condition in $Q_{\bar{C}, F, g}$.
Let $X$ be a countable elementary substructure of $H(\beth_{3}^{+})$ with $p_{0},\bar{C}, F$ and $g$ in $X$.
Let $\gamma = X \cap \omega_{1}$ and let $y$ be an element of $2^{\omega}$ extending
\[\langle p_{0}(C_{\gamma}(i)) : i < \omega,\, C_{\gamma}(i) < \delta_{p_{0}}\rangle\] with $F(y) = g(\gamma)$. Let $R$ be the set of $q \leq p_{0}$ such that \[\langle q(C_{\gamma} (i)) : i < \omega,\, C_{\gamma}(i) < \delta_{q}\rangle\] is an initial segment of $y$.

We claim that for each dense subset $D$ of $Q_{\bar{C}, F, g}$ in $X$ and each $p \in R \cap X$, there is a $q \leq p$ in $R \cap D \cap X$. To see that the claim holds, fix $D$ and $p$ and let $Y \in X$ be a countable elementary substructure of $H((\beth_{2})^{+})$ such that $D,p \in Y$.
Extend $p$ to a condition $p' \in R \cap Y$ with $C_{\gamma} \cap Y \subseteq \delta_{p'}$. Then let $q \leq p'$ be an element of $Y \cap D$. It follows from the claim that there exists a condition below $p_{0}$ which is in each dense open subset of $Q_{\bar{C}, F, g}$ in $X$.
\end{proof}

\begin{remark}\label{aleph1Borelrem} In the forcing extension we produce, each subset of $\breals$ which is either universally measurable or universally categorical
will be a union of $\aleph_{1}$ many Borel sets.
Lemma \ref{codingitlem} below will be used to show that in addition, in this extension, every subset of $\cP(\omega)$ of cardinality $\aleph_{1}$ is $\uTSigma^{1}_{2}$. Analytic subsets of $\breals$ are naturally coded by elements of $\cP(\omega)$ (coding trees; see for instance Section 27A of \cite{Ke}) in such a way that the set of pairs $(x,y)$ such that $x \subseteq \omega$, $y \in \breals$ and $x$ is in the set coded by $y$ is analytic. From this one gets that, in our extension, all subsets of $\breals$ which are either universally measurable or universally categorical
are $\uTSigma^{1}_{2}$. Since the set of universally measurable sets and the set of universally categorical sets are both closed under complements, it follows that, in this extension, each such set is $\uTDelta^{1}_{2}$.
\end{remark}

Our coding of elements of $[\cP(\omega)]^{\aleph_{1}}$ uses certain iterations of length $\omega$ of partial orders of the form
$Q_{\bar{C}, F, g}$, which we now define. Let $\pi \colon \omega \times \omega \to \omega$ be a fixed recursive bijection, and let $\pi_{0}$ and $\pi_{1}$ be functions from $\omega$ to $\omega$ such that $\pi(i) = (\pi_{0}(i), \pi_{1}(i))$ for  all $i <  \omega$. We choose $\pi$ so that $\pi_{0}(i) < i$ for all $i > 0$. We define, for each triple $(\bar{C}, F, g)$ such that
\begin{itemize}
\item $\bar{C}$ is a ladder system on $\omega_{1}$,
\item $F$ is a dense function from $2^{\omega}$ to $2$ and
\item $g$ is a function from $\Omega$ to $2$
\end{itemize}
the following objects recursively on $i <\omega$ (and suppress discussion of the meaning of
``canonical name", trusting the reader to supply her or his preferred definition). Let
\begin{itemize}
\item $P_{0}$ be the trivial partial order;
\item $\dot{g}_{0}$ be the canonical $P_{0}$-name for $g$;
\item $\dot{Q}_{0}$ be the canonical $P_{0}$-name for $Q_{\bar{C}, F, g}$;
\item for all $i < \omega$,
\begin{itemize}
\item $P_{i +1}$ be $P_{i} * \dot{Q}_{i}$;
\item $\dot{h}_{i}$ be the canonical $P_{i+1}$-name for the $\dot{Q}_{i}$-generic function from $\omega_{1}$ to $2$;
\end{itemize}
\item for all positive $i <\omega$,
\begin{itemize}
\item  $\dot{g}_{i}$ be the canonical $P_{i}$-name for the set of pairs $(\alpha, k)$ such that $\alpha \in \Omega$ and $k = \dot{h}_{\pi_{0}(i), G_{i}}(\alpha + \pi_{1}(i))$, where $G_{i}$ denotes the generic filter for $P_{i}$;
\item $\dot{Q}_{i}$ be the canonical $P_{i}$-name for $Q_{\bar{C}, F, \dot{g}_{i}}$.
\end{itemize}
\end{itemize}
We then let $Q^{*}_{\bar{C}, F, g}$ denote the full (i.e., countable) support limit of the forcing iteration $\langle P_{i}, \dot{Q}_{i} : i < \omega \rangle$.

The purpose of this definition is given in Lemma \ref{codingitlem} below, but the key point is that, for any $\alpha \in \Omega$, the values $g_{i}(\alpha)$ ($i < \omega$) determined via $\pi$ the values $h_{i}(\alpha + n)$ ($i,n \in \omega$), and similarly, for any $\beta \in \Omega$ and the values of $h_{i}(\alpha)$ ($i \in \omega, \alpha < \beta$) determine the values $g_{i}(\beta)$ ($i < \omega$) via $F$. So relative to the ground model objects $F$ and $\pi$, the sequence $\langle h_{i} \restrict \omega : i < \omega \rangle$ determines the sequence $\langle (g_{i}, h_{i}) : i < \omega \rangle$.


Given $a \subseteq \omega$, we define the \emph{canonical ladder system relative to} $a$ to be the set $\bar{C}_{a} = \{ C^{a}_{\alpha} : \alpha < \Omega \cap \omega_{1}^{\bL[a]}\},$ where each $C^{a}_{\alpha}$ is the constructibly least (in $\bL[a]$, relative to $a$) cofinal subset of the corresponding $\alpha$ of ordertype $\omega$. This defines a ladder system in $\bL[a]$ which is a ladder system in $\bV$ if and only if $\omega_{1}^{\bL[a]} = \omega_{1}$. Let $\Omega'$ be the set of countable limits of limit ordinals.

\begin{lem}\label{codingitlem} Let $a$ be a subset of $\omega$ such that $\omega_{1}^{\bL[a]} = \omega_{1}$, and
let
$F \colon 2^{\omega} \to 2$ be a dense function whose graph is $\uTSigma^{1}_{2}$ in $\bL[a]$. Let $B$ be a subset of $\cP(\omega)$ of cardinality at most $\aleph_{1}$. There exists a function $g \colon \Omega \to 2$ such that, whenever
$G$ is a $V$-generic filter for $Q^{*}_{\bar{C}_{a}, F, g}$ and $M$ is an outer model of $\bV[G]$, then $B$ is
$\uTSigma^{1}_{2}$ in $M$.
\end{lem}

\begin{proof}
It suffices to consider the case where $B$ is nonempty.
We fix (and suppress) a $\uTSigma^{1}_{2}$ definition for $F$, and note that, by Shoenfield absoluteness, in any outer model this definition will define a partial function from $\breals$ to $2$ containing $F$.

Let $g \colon \Omega \to 2$ be such that \[B = \{ \{ k \in \omega : g(\beta + \omega \cdot k) = 0\} : \beta \in \Omega' \},\]
let $G$ be $\bV$-generic for $Q^{*}_{\bar{C}_{a}, F, g}$ and let $M$ be an outer model of $\bV[G]$. For each $i < \omega$, let $h_{i}$ be $\dot{h}_{i, G \restrict P_{i+1}}$, where $\dot{h}_{i}$ is as in the definition of $Q^{*}_{\bar{C}, F, g}$. We show that, in $M$, $B$ is $\Sigma^{1}_{2}$ in $a$ and $\langle h_{i} \restrict \omega : i < \omega \rangle$.
  In particular, $B$ is the set of $x \subseteq \omega$ such that there exist
  \begin{itemize}
  \item an element $y$ of $\breals$ coding a model of the form $\bL_{\alpha}[a]$, for some countable ordinal $\alpha$ (the wellfoundedness of this model being a $\Pi^{1}_{1}$ condition on $y$), such that $\omega_{1}^{\bL_{\alpha}[a]}$ exists (i.e., some element of $\alpha$ is uncountable in $\bL_{\alpha}[a]$),
  \item $\beta \in \Omega' \cap \omega_{1}^{\bL_{\alpha}[a]}$,
  \item functions $h^{*}_{i} \colon \beta + \omega \cdot \omega \to 2$ ($i < \omega$) such that, for each $i < \omega$, $h^{*}_{i} \restrict \omega =
  h_{i} \restrict \omega$ and
  \item functions $g^{*}_{i} \colon \Omega \cap (\beta + \omega \cdot \omega) \to 2$ ($i < \omega$)
  \end{itemize}
such that,
letting $F^{*}$ be the function computed in $\bL_{\alpha}[a]$ using our $\uTSigma^{1}_{2}$ definition for $F$,
\begin{enumerate}
\item\label{liftitem}  for each $i< \omega$ and each $\gamma \in \Omega \cap (\beta + \omega \cdot \omega)$,
\begin{itemize}
\item $\langle h^{*}_{i}(C^{a}_{\gamma}(j)) : j < \omega \rangle$ is in the domain of $F^{*}$;
\item $g^{*}_{i}(\gamma) = F^{*}(\langle h^{*}_{i}(C^{a}_{\gamma}(j)) : j < \omega \rangle)$;
\item if $i > 0$ then $g^{*}_{i}(\gamma) = h^{*}_{\pi_{0}(i)}(\gamma + \pi_{1}(i))$;
\end{itemize}
\item $x = \{ k \in \omega : g^{*}_{0}(\beta + \omega \cdot k) = 0\}$.

\end{enumerate}
That there exist such objects for each element of $B$ follows from the choice of $g$ and the definition of $Q^{*}_{\bar{C}_{a}, F, g}$. Given such objects, the choice of the names $\dot{g}_{i}$ (and item (\ref{liftitem}) above) then implies (via an inductive proof on
$\gamma \in \Omega \cap (\beta + \omega\cdot\omega)$) that for each $i < \omega$, $g^{*}_{i}  = \dot{g}_{i, G \restrict P_{i}}\restrict (\beta + \omega \cdot \omega)$ and
$h^{*}_{i} = h_{i} \restrict (\beta + \omega \cdot \omega)$, which, again by the choice of $g$, implies that the corresponding set $x$ is in $B$.
\end{proof}



\section{Sequences and trees}\label{seqtreesec}

Remarks \ref{reducerem} and \ref{popropsrem} and Lemma \ref{codingitlem} reduce the proof of Theorem \ref{specthrm} to showing that the forcing
iterations we consider are $\cA$-representing, where $\cA$ is either the collection of universally measurable sets or the collection of universally categorical sets. These iterations will be countable support iterations where each successor step is a partial order of the form $Q_{\bar{C}, F, \dot{g}}$, for a fixed ladder system $\bar{C}$ on $\omega_{1}$ and a fixed dense function $F$ (moreover, our iterations will consist of segments of the form $Q^{*}_{\bar{C}, F, \dot{g}}$). An additional requirement on $F$ (total pathology) will be introduced in Section \ref{pathsec}. We fix the following notation.


\begin{df}\label{QCFdef} Let $\bar{C}$ be a ladder system on $\omega_{1}$
and let $F \colon 2^{\omega} \to 2$ be a dense partial function.
Let $\bQ_{\bar{C}, F}$ be the class of sequences of the form
\[ \langle \bbP_{\alpha}, \dot{Q}_{\beta}, \dot{g}_{\beta}, \dot{h}_{\beta} : \alpha \leq \alpha_{*}, \beta < \alpha_{*} \rangle,\] where
\begin{itemize}
\item $\alpha_{*}$ is an ordinal,
\item $\langle \bbP_{\alpha}, \dot{Q}_{\beta} : \alpha \leq \alpha_{*}, \beta < \alpha_{*} \rangle$ is a countable support iteration,
\item each $\dot{g}_{\beta}$ is a $P_{\beta}$-name for a function from $\Omega$ to $2$,
\item each $\dot{Q}_{\beta}$ is a $P_{\beta}$-name for the partial order $Q_{\bar{C}, F, \dot{g}_{\beta}}$,
\item each $\dot{h}_{\beta}$ is a $P_{\beta + 1}$-name for the $\dot{Q}_{\beta}$-generic function.
\end{itemize}
We say that $\langle \bbP_{\alpha}, \dot{Q}_{\beta}, \dot{g}_{\beta}, \dot{h}_{\beta} : \alpha \leq \alpha_{*}, \beta < \alpha_{*} \rangle \in \bQ_{\bar{C}, F}$ is \emph{fully bookkeeping} if for each $\alpha < \alpha_{*}$ and each $\bbP_{\alpha}$-name $\dot{B}$ for a nonempty set of reals of cardinality at most $\aleph_{1}$, there is a $\alpha' \in [\alpha, \alpha_{*})$ such that $1_{\bbP_{\alpha'}}$ forces that the iteration induced by $\langle \dot{Q}_{\alpha' + n} : n \in \omega \rangle$ will be a forcing of the form $Q^{*}_{\bar{C}, F, g}$ (as computed in the $\bbP_{\alpha'}$-extension), with $g$ chosen as in Lemma \ref{codingitlem} so that $Q^{*}_{\bar{C}, F, g}$ forces $B$ to be $\uTSigma^{1}_{2}$ in its extension and in the $\bbP_{\alpha_{*}}$ extension.
\end{df}




In Section \ref{itsec} we will be building a suitable $X$-generic condition for some $\bbP_{\alpha_{*}}$ as above, where $X$ is a countable elementary submodel of a large enough initial segment of the universe containing our iteration. The first definition below lists useful data that comes with such a situation. The second definition
describes a tree of conditions which will guide us to finding our desired condition.


\begin{df}\label{sdsdef}
  A \emph{suitable data sequence} is sequence
  \[\langle \bar{C}, F, \bq, p, X, \bar{Y}, \bar{\i}, \bar{D} \rangle\]
  such that
\begin{itemize}
\item $\bar{C} = \langle C_{\alpha} : \alpha \in \Omega \rangle$ is a ladder system on $\omega_{1}$;
\item $F \colon 2^{\omega} \to 2$ is a dense function;
\item $\bq = \langle \bbP_{\alpha}, \dot{Q}_{\beta}, \dot{g}_{\beta}, \dot{h}_{\beta} : \alpha \leq \alpha_{*}, \beta < \alpha_{*} \rangle$ is in $\bQ_{\bar{C}, F}$;
\item $p$ is in $\bbP_{\alpha_{*}}$;
\item $X \prec (H((2^{2^{|\bq|}})^{+}), \in)$ is countable, with $\{ \bar{C}, F, \bq, p\} \in X$;
\item $\bar{Y} = \langle Y_{k} : k \in \omega \rangle$ is an $\in$-chain of countable elementary substructures of
$H((2^{|\bq|})^{+})$, such that
\begin{itemize}
\item $\{ \bar{C}, F, \bq, p\} \in Y_{0}$,
\item $X \cap H((2^{|\bq|})^{+}) = \bigcup_{k \in \omega}Y_{k}$,
\item for all $k \in \omega$, $C_{X \cap \omega_{1}} \cap (Y_{k+1} \setminus Y_{k}) \neq \emptyset$;
\end{itemize}
\item $\bar{\i} = \langle i_{k} : k \in \omega \rangle$ lists $\ot(X \cap \alpha_{*})$ without repetition in such a way that for each $k \in \omega$ the $i_{k}$th element of $X \cap \alpha_{*}$ is in $Y_{k}$;
\item $\bar{D} = \langle D_{k} : k \in \omega \rangle$ lists the dense open subsets of $\bbP_{\alpha_{*}}$ in $X$ in such a way that each $D_{k}$ is in the corresponding $Y_{k}$.
\end{itemize}
\end{df}

In the following definition, note that by the conditions on $\bar{\i}$ in Definition \ref{sdsdef}, $\alpha_{i_{k}}$ is in $Y_{k}$, for each $k \in \omega$.

\begin{df}\label{condtreedef}
  Suppose that
  \[S = \langle \bar{C}, F, \bq, p, X, \langle Y_{k} : k \in \omega \rangle, \langle i_{k} : k \in \omega \rangle, \langle D_{k} : k \in \omega \rangle \rangle\]
  is a suitable data sequence, and let
    \begin{itemize}
  \item $\bq$ be $\langle \bbP_{\alpha}, \dot{Q}_{\beta}, \dot{g}_{\beta}, \dot{h}_{\beta} : \alpha \leq \alpha_{*}, \beta < \alpha_{*} \rangle$,
  \item $\gamma$ be $X \cap \omega_{1}$,
  \item $\bar{\alpha} = \langle \alpha_{i} : i < i_{*} \rangle$ list $X \cap \alpha_{*}$ in increasing order,
  \item $\bar{C}$ be $\langle C_{\alpha} : \alpha \in \Omega\rangle$,
  \item for each $k \in \omega$, $m_{k}$ be $|C_{\gamma} \cap Y_{k}|$ and $u_{k}$ be $\{ i_{n} : n < k \}$,
  \item $\alpha_{i_{*}}$ denote $\alpha_{*}$ and, for each $j \leq i_{*}$ and $k \in \omega$, $u_{k}(j)$ be $u_{k} \cap j$.
  \end{itemize}
A \emph{condition tree} for $S$ is a pair $(T, \langle p_{\rho} : \rho \in T \rangle)$
  such that

\begin{enumerate}
\item\label{itemone} $T$ is a finitely branching tree of finite sequences, with a unique node of length $1$;
\item\label{itemtwo} each $\rho \in T$ is sequence of the form $\langle s^{\rho}_{k} : k < |\rho| \rangle$
such that
\begin{enumerate}
\item each $s^{\rho}_{k}$ is a function from $u_{k}$ to $2^{m_{k}}$;
\item\label{itemtwob} whenever $k + 1 < |\rho|$ and $i \in u_{k}$, $s^{\rho}_{k}(i) \trianglelefteq s^{\rho}_{k+1}(i)$;
\end{enumerate}
\item\label{itemthree} whenever $\rho \in T$ has length $k+1$ and $s \colon u_{k} \to 2^{m_{k+1}}$ is such that,
for all $i \in u_{k}$,
$s^{\rho}_{k}(i) \trianglelefteq s(i)$,
there is exactly one $\rho' \in T$ such that:
\begin{itemize}
\item $\rho \trianglelefteq \rho'$,
\item $|\rho'| = k+2$,
\item $s \subseteq s^{\rho'}_{k + 1}$;
\end{itemize}
\item\label{itemthreepfive} whenever $\rho_{1}, \rho_{2} \in T$ have length at least $k + 2$ and $j < i_{*}$ is such that
\[s^{\rho_{1}}_{k+1} \restrict u_{k}(j) = s^{\rho_{2}}_{k+1} \restrict u_{k}(j),\]
then we have
\[s^{\rho_{1}}_{k+1} \restrict u_{k+1}(j) = s^{\rho_{2}}_{k+1} \restrict u_{k+1}(j);\]
\item\label{belp} $p_{\langle\rangle}  = p$;
\item\label{gooddec} for all $\rho \in T\setminus \{\langle \rangle\}$, $p_{\rho} \in D_{|\rho|-1} \cap Y_{|\rho|-1}$;

\item if $\rho_{1} \trianglelefteq \rho_{2} \in T$ then $p_{\rho_{1}} \geq_{\bbP_{\alpha_{*}}} p_{\rho_{2}}$;

\item\label{itemseven} if $\rho \in T$ and $|\rho| = k + 1$ then for all $i \in u_{k}$,
\[p_{\rho}\restrict (\alpha_{i}+1)\Vdash_{\bbP_{\alpha_{i}+1}} \langle \dot{h}_{\alpha_{i}}(C_{\gamma}(n)) : n < m_{k}
\rangle =  \check{s}^{\rho}_{k}(i)\]
and $p_{\rho}\restrict \alpha_{i}$ decides the value of $\delta_{p_{\rho}(\alpha_{i})}$;

\item\label{itemeight} if $\alpha \leq \alpha_{*}$ and $\rho_{1}, \rho_{2} \in T$ have the same length and are such that
\[s^{\rho_{1}}_{k}(i) = s^{\rho_{2}}_{k}(i)\] whenever $k < |\rho_{1}|$, $i \in u_{k}$ and $\alpha_{i} <  \alpha$,
then $p_{\rho_{1}} \restrict \alpha = p_{\rho_{2}} \restrict \alpha$.


\end{enumerate}
\end{df}



\begin{remark}
Using parts (\ref{itemone}) and (\ref{itemseven}) of Definition \ref{condtreedef}, the conjunction of parts (\ref{itemthreepfive}) and (\ref{itemeight}) can be rephrased as the following statement : whenever $\rho_{1}$ and $\rho_{2}$ are distinct members of $T$ of the same length $k + 2$, letting $i$ be the least member of $u_{k+1}$ such that $s^{\rho_{1}}_{k+1}(i) \neq s^{\rho_{2}}_{k+1}(i)$, we have that $i \in u_{k}$ and $\alpha_{i}$ is the least $\alpha$ such that $p_{\rho_{1}} \restrict (\alpha + 1) \neq p_{\rho_{2}} \restrict (\alpha + 1)$.
\end{remark}

The following lemma is an adaptation of the proof of Lemma 1.8 of \cite{Sh64}.

\begin{lem}\label{fusionlemma}
If $S$ is a suitable data sequence, then there exists a condition tree relative to $S$.
\end{lem}

%

\begin{proof} Let $S$ be
\[\langle \bar{C}, F, \bq, p, X, \langle Y_{k} : k \in \omega \rangle, \langle i_{n} : n < \omega \rangle, \langle D_{k} : k \in \omega \rangle \rangle,\]
and let
  \begin{itemize}
  \item $\gamma$ be $X \cap \omega_{1}$,
  \item $\bq$ be $\langle \bbP_{\alpha}, \dot{Q}_{\beta}, \dot{g}_{\beta}, \dot{h}_{\beta} : \alpha \leq \alpha_{*}, \beta < \alpha_{*} \rangle$,
  \item $\bar{\alpha} = \langle \alpha_{i} : i < i_{*} \rangle$ list $X \cap \alpha_{*}$ in increasing order,
  \item $\bar{C}$ be $\langle C_{\alpha} : \alpha \in \Omega\rangle$ and
  \item for each $k \in \omega$, $m_{k}$ be $|C_{\gamma} \cap Y_{k}|$ and $u_{k}$ be $\{ i_{n} : n < k\}$.
  \end{itemize}
 We build $T$ and $\langle p_{\rho} : \rho \in T\rangle$ by recursion on the length of $\rho \in T$.  For each $k \in \omega$, let $T(k)$ denote the set of sequences in $T$ of length $k$. For $k = 0$ we let $p_{\langle \rangle} = p$.



 For each $k \in \omega$, let $B_{k}$ be the (nonempty) set of conditions $q \leq p$ in $D_{k} \cap Y_{k}$ such that
 for each $i \in u_{k}$,
\begin{itemize}
\item  $q \restrict \alpha_{i}$ decides the value of $\delta_{q(\alpha_{i})}$ (necessarily to be in $Y_{k} \cap \omega_{1}$);
\item\label{uniqtwo} $q \restrict (\alpha_{i} + 1)$ decides the value of $\langle \dot{h}_{\alpha_{i}}(C_{\gamma}(n)) : n <  m_{k} \rangle$.
\end{itemize}
Given $q \in \bbP_{\alpha_{*}}$, $i < i_{*}$ and $k \in \omega$ such that $q$ decides the value
of the sequence $\langle \dot{h}_{\alpha_{i}}(C_{\gamma}(n)) : n <  m_{k} \rangle$, let $t^{q}_{i,k}$ denote this decided value.
Each $\rho \in T$ of length greater than zero will be determined by the corresponding $p_{\rho}$. That is, conditions (\ref{gooddec}) and (\ref{itemseven}) in Definition \ref{condtreedef} imply that each $p_{\rho}$ is in $B_{|\rho|-1}$ and each value $s^{\rho}_{\ell}(i)$ ($\ell < |\rho|$, $i \in u_{\ell}$) is the corresponding $t^{p_{\rho}}_{i,\ell}$.

To start, let $q_{1}$ be any member of $B_{0}$.
Let $\rho_{1}$ be the sequence of length $1$ whose only member is the emptyset (noting that $u_{0}$ is the emptyset).
Let $T(1) = \{ \rho_{1}\}$ and let $p_{\rho_{1}}$ be $q_{1}$.

Suppose now that $k \in \omega \setminus \{0\}$ and that the members of $T(k)$ and the conditions $p_{\rho} \in B_{k-1}$ $(\rho \in T(k))$ have been chosen so as to satisfy the conditions in Definition \ref{condtreedef}.
Let $Z$ be the set of pairs $(\rho, s)$, where $\rho \in T(k)$ and  $s \colon u_{k-1} \to 2^{m_{k}}$ are  as in item (\ref{itemthree}) of Definition \ref{condtreedef}, that is, such that for all $i \in u_{k-1}$, $s^{\rho}_{k-1}(i) \trianglelefteq s(i)$.
Working in $Y_{k}$ we pick for each pair $(\rho, s) \in Z$ a condition $p_{\rho, s} \leq p_{\rho}$ in $B_{k}$ meeting conditions (\ref{itemseven}) and (\ref{itemeight}) of Definition \ref{condtreedef} and such that the induced sequences $t^{p_{\rho, s}}_{i,k}$ ($i \in u_{k}$) will satisfy conditions (\ref{itemthree}) and (\ref{itemthreepfive}) of Definition \ref{condtreedef}. Then the set of conditions corresponding to the members of $T(k+1)$ will be $\{ p_{\rho, s} : (\rho, s) \in Z\}$, and this will induce the set $T(k+1)$.

We will do this in three stages. In the first stage we (eventually) choose conditions $p^{1,\max(u_{k}),|Z|}_{\rho, s} \leq p_{\rho}$ for each pair $(\rho,s) \in Z$, in such a way that the sequences $t^{p_{\rho,s}}_{i,k}$ are all determined (for each $(\rho, s)\in Z$ and $i\in u_{k}$, assuming, as will be the case, that the eventually chosen condition $p_{\rho, s}$ will be a strengthening of $p^{1,\max(u_{k}), |Z|}_{\rho, s}$).
In the second stage we strengthen each condition $p^{1,\max(u_{k}),|Z|}_{\rho, s}$ to a condition $p^{2,|Z|}_{\rho, s}$ in $D_{k}$. In the third stage each condition $p^{2,|Z|}_{\rho, s}$ is strengthened to a condition $p^{3,\min(u_{k}),|Z|}_{\rho, s} = p_{\rho, s}$ deciding the value of $\delta_{p_{\rho,s}}(\alpha_{i})$ for each $i \in u_{k}$. In each stage the choices are made simultaneously for all $(\rho, s) \in Z$ order to ensure that conditions (\ref{itemthreepfive}) and (\ref{itemeight}) of Definition \ref{condtreedef} are preserved.



In the first stage we choose conditions of the form $p^{1,i,\ell}_{\rho, s} \leq p_{\rho}$ $(\ell \leq |Z|, i \in u_{k})$,
recursively first in $u_{k}$ (simultaneously for all
$(\rho, s) \in Z$) and secondarily in $|Z|$ (although this is needed only the case $i = i_{k-1}$),
in such a way that
\begin{enumerate}
\item\label{rhoscondone} for all $(\rho, s) \in Z$,
\begin{enumerate}
\item $p^{1,\min(u_{k}),0}_{\rho, s} = p_{\rho}$;
\item for all $i \in u_{k}$ and $\ell < |Z|$, $p^{1,i,\ell+1}_{\rho, s} \leq p^{1,i,\ell}_{\rho,s}$;
\item\label{rhoscondonec} for all $i \in u_{k}$, $\ell \leq |Z|$ and $\alpha \in \alpha_{*} \setminus (\alpha_{i} + 1)$, $p^{1,i,\ell}_{\rho,s}(\alpha) = p_{\rho}(\alpha)$;
\item when $i < i'$ are successive members of $u_{k}$, $p^{1,i,|Z|}_{\rho, s} = p^{1,i',0}_{\rho, s}$;
\item\label{rhoscondoneb} for each $i \in u_{k-1}$, $p^{1,i,1}_{\rho, s} = p^{1,i,|Z|}_{\rho, s}$ and \[p^{1,i,1}_{\rho, s}\restrict (\alpha_{i}+1)\Vdash_{\bbP_{\alpha_{i}+1}}\langle \dot{h}_{\alpha_{i}}(C_{\gamma}(n)) : n < m_{k}\rangle = s(i);\]
\item\label{rhoscondoned}
$p^{1,i_{k-1},|Z|}_{\rho, s}\restrict (\alpha_{i_{k-1}}+1)$ decides the value of
\[\langle \dot{h}_{\alpha_{i_{k-1}}}(C_{\gamma}(n)) : n < m_{k}\rangle.\]
\end{enumerate}
\item\label{rhoscondtwo} for all $(\rho_{1}, s_{1}), (\rho_{2}, s_{2}) \in Z$, all $i \in u_{k}$ all $\ell \leq |Z|$ and all $\alpha \leq \alpha_{*}$, if
\[s_{1} \restrict \{ j \in u_{k-1} : \alpha_{j} < \alpha \} = s_{2} \restrict \{ j \in u_{k-1} : \alpha_{j} < \alpha \}\]
then $p^{1,i,\ell}_{\rho_{1}, s_{1}} \restrict \alpha  = p^{1,i,\ell}_{\rho_{2}, s_{2}} \restrict \alpha $.
\end{enumerate}
Note that item (\ref{rhoscondtwo}) above corresponds to condition (\ref{itemeight}) of Definition \ref{condtreedef}, which holds for the conditions $p_{\rho}$ for $\rho \in T(k)$ by the recursive hypothesis.

We choose the conditions $p^{1,i,\ell}_{\rho, s}$ ($(\rho, s) \in Z$, $1 \leq \ell \leq |Z|$) as follows. For each $(\rho, s) \in Z$, let $p^{1,i,0}_{\rho, s}$ be $p_{\rho}$ if $i = \min(u_{k})$. Otherwise let $p^{1,i,0}_{\rho, s}$ be $p^{1,i',|Z|}_{\rho, s}$, where $i'$ is the largest member of $u_{k}$ below $i$. Enumerate $|Z|$ as $\langle (p_{\ell}, s_{\ell}) : \ell  < |Z| \rangle$.

First consider the case $i \in u_{k-1}$. Define an equivalence relation $E_{i}$ on $Z$ consisting of the pairs $((\rho,s),(\rho',s'))$ such that
\[s \restrict \{ j \in u_{k-1} : \alpha_{j} \leq  \alpha_{i}\} = s'\restrict \{ j \in u_{k-1} : \alpha_{j} \leq  \alpha_{i}\}.\] By the recursive hypothesis applied to item (\ref{rhoscondtwo}), for any two $E_{i}$-equivalent pairs $(\rho, s)$ and $(\rho', s')$ in $Z$,
\[p^{1,i,0}_{\rho, s} \restrict (\alpha_{i} + 1) = p^{1,i,0}_{\rho',s'} \restrict (\alpha_{i}+1).\]
Each condition $p^{1,i,0}_{\rho,s}$ is below $p_{\rho}$, which forces that $\delta_{p_{\rho}(\alpha_{i})} < Y_{k-1} \cap \omega_{1}$. Furthermore, $Y_{k-1} \cap \omega_{1}$ is below the least element of $C_{\gamma}$ for which the value of $\dot{h}_{\alpha_{i}}$ is not decided by $p_{\rho}$.
By item (\ref{rhoscondonec}) (applied to the previous values of $i$, if there are any), $p^{1,i,0}_{\rho, s}(\alpha_{i}) = p_{\rho}(\alpha_{i})$ for all $(\rho, s) \in Z$. We can then strengthen each condition $\rho^{1,i,0}_{(\rho, s)}$ $((\rho, s) \in Z)$ to a condition $\rho^{1,i,1}_{(\rho, s)}$ while changing only coordinate $\alpha_{i}$
(in the same way for $E_{i}$-equivalent pairs) to make the displayed formula from item (\ref{rhoscondoneb}) hold.
As in item (\ref{rhoscondoneb}), for all $\ell \in (1,|Z|]$, we let $p^{1,i,\ell}_{\rho, s} = p^{1,i,1}_{\rho,s}$.

When $i \in u_{k} \setminus u_{k-1}$ (i.e., $i = i_{k-1}$), our task is easier in that there is are no previously specified values of $\langle \dot{h}_{\alpha_{i}}(C_{\gamma}(n)) : n < m_{k}\rangle$ to obtain. However, it is harder in that we may need to change our conditions in coordinates below $\alpha_{i}$ while preserving item (\ref{rhoscondtwo}) at lower values of $\alpha$ (conditions disagreeing on the interval $[0,\alpha_{i}]$ may still agree on some such interval $[0,\alpha)$). Our method for doing this will be used again in stages two and three.

For each $\ell < |Z|$, given the conditions $p^{1,i,\ell}_{\rho, s}$ we
strengthen condition $p^{1,i,\ell}_{\rho_{\ell}, s_{\ell}}$ (perhaps trivially, and only in coordinates less than or equal to $\alpha_{i}$) to decide the value of
$\langle \dot{h}_{\alpha_{i}}(C_{\gamma}(n)) : n < m_{k}\rangle$, thus forming condition $p^{1,i,\ell+1}_{\rho_{\ell}, s_{\ell}}$ (still in $Y_{k}$). For each other condition $p^{1,i,\ell}_{\rho, s}$ (i.e., with $(\rho, s) \in Z \setminus \{(\rho_{\ell}, s_{\ell})\}$), letting $\alpha$ be maximal with 
\[s \restrict \{j \in u_{k-1} : \alpha_{j} < \alpha\} = s_{\ell} \restrict \{ j \in u_{k-1} : \alpha_{j} < \alpha\},\]
(by item (\ref{rhoscondtwo}) this implies that $p^{1,i,\ell}_{\rho, s} \restrict \alpha = p^{1,i,\ell}_{\rho_{\ell}, s_{\ell}} \restrict \alpha$,)
we let $p^{1,i,\ell + 1}_{\rho,s}$ be the concatenation of $p^{1,i,\ell+1}_{\rho_{\ell},s_{\ell}} \restrict \alpha$ with $p^{1,i,\ell}_{\rho, s} \restrict [\alpha, \alpha_{*})$.


The choice of the conditions $p^{1,\max(u_{k}),|Z|}_{\rho, s}$ induces the corresponding members of $T(k+1)$: letting, for each $(\rho, s) \in Z$, $t_{\rho, s} \colon u_{k} \to 2^{m_{k}}$ be such that \[t_{\rho, s}(i) = t^{p^{1,\max(u_{k}),|Z|}_{\rho,s}}_{i,k},\]
we let $T(k+ 1)$ be \[\{ \rho^{\frown}\langle t_{\rho, s} \rangle : (\rho, s) \in Z\}.\]
The members of $T(k+1)$ have been chosen 
to satisfy conditions (\ref{itemthree}) and (\ref{itemthreepfive}) of Definition \ref{condtreedef}.

In stages 2 and 3 we will strengthen each $p^{1,\max(u_{k}),|Z|}_{\rho, s}$ to a condition $p_{\rho, s}$ by modifications of the argument above for item (\ref{rhoscondoned}). We will then let $p_{\rho, s}$ be 
$p_{\rho^{\frown}\langle t_{\rho, s} \rangle}$ for each $(\rho, s) \in Z$.
Condition (\ref{rhoscondoneb}) above implies that any strengthening of any $p^{1,\max(u_{k}), |Z|}_{\rho, s}$ will satisfy condition (\ref{itemseven}) of Definition \ref{condtreedef}, aside from the requirement that the values of $\delta_{p_{\rho, s}(\alpha_{i})}$ ($i \in u_{k}$) must be decided, which we will take care of in stage 3. 

The recursive hypotheses and item (\ref{rhoscondtwo}) above imply that the conditions $p^{1,\max(u_{k}),|Z|}_{\rho, s}$ do not violate condition (\ref{itemeight}) of Definition \ref{condtreedef}. In particular, we have that for all distinct pairs $(\rho, s), (\rho', s')$ from $Z$, the greatest $\alpha \leq \alpha_{*}$ for which $p^{1,\max(u_{k}), |Z|}_{\rho,s} \restrict \alpha = p^{1,\max(u_{k}),|Z|}_{\rho', s'} \restrict \alpha$ is also the least value $\alpha_{i}$ ($i \in u_{k-1}$) for which $s(i) \neq s'(i)$. We will maintain this equivalence throughout the rest of the construction for the various conditions associated to the pairs $(\rho, s)$ from $Z$.


We now (in stage 2) strengthen each $p^{1,\max(u_{k}),|Z|}_{\rho, s}$ to a condition $p^{2, |Z|}_{\rho, s}$ in $D_{k}$, while respecting condition (\ref{itemeight}) of Definition \ref{condtreedef}. We reuse our enumeration of $|Z|$ as $\langle (p_{\ell}, s_{\ell}) : \ell  < |Z| \rangle$.
We choose conditions $p^{2,\ell}_{\rho, s}$
for $(\rho, s) \in Z$ and $\ell \leq |Z|$, recursively in $\ell$ so that:
\begin{enumerate}
\setcounter{enumi}{2}
\item for all $(\rho, s) \in Z$, 
   \begin{enumerate}
      \item $p^{2,0}_{\rho, s} = p^{1,\max(u_{k}),|Z|}_{\rho, s}$;
      \item for all $\ell < |Z|$, $p^{2,\ell + 1}_{\rho, s} \leq p^{2,\ell}_{\rho, s}$ and $p^{2, \ell + 1}_{\rho_{\ell}, s_{\ell}} \in B_{k}$;
   \end{enumerate}
\item\label{coherencec} for all $\ell \leq |Z|$, and all distinct $(\rho, s), (\rho', s') \in Z$, the least $\alpha$ such that
\[p^{2,\ell}_{\rho, s}(\alpha) \neq p^{2,\ell}_{\rho', s'}(\alpha)\] is the same as the least $\alpha$ such that
\[p^{2,0}_{\rho, s}(\alpha) \neq p^{2,0}_{\rho', s'}(\alpha).\]
\end{enumerate}
Item (\ref{coherencec}) ensures the preservation of condition (\ref{itemeight}) of Definition \ref{condtreedef}, i.e., that, for all $(\rho, s), (\rho', s') \in Z$ and all $\alpha \leq \alpha_{*}$,
if
\[s\restrict \{ j \in u_{k} : \alpha_{j} <\alpha \} = s'\restrict \{ j \in u_{k} : \alpha_{j} <\alpha\}\]
then $p^{2,\ell}_{\rho, s} \restrict \alpha = p^{2,\ell}_{\rho', s'} \restrict \alpha$.

In substage $\ell$ of stage 2, strengthen $p^{2,\ell}_{\rho_{\ell}, s_{\ell}}$ (possibly trivially) to a condition $p^{2,\ell+1}_{\rho_{\ell},s_{\ell}}$ in $B_{k}$.
For each of the other conditions $p^{2,\ell}_{\rho, s}$, letting $\alpha$ be minimal such that \[p^{2,\ell}_{\rho_{\ell}, s_{\ell}}(\alpha) \neq p^{2,\ell}_{\rho, s}(\alpha)\]
let $p^{2,\ell + 1}_{\rho, s}$ be the concatenation of $p^{2,\ell + 1}_{\rho_{\ell}, s_{\ell}} \restrict \alpha$ with $p^{2,\ell}_{\rho, s} \restrict [\alpha, \alpha_{*})$.

Finally, in stage 3, working recursively through the members of $u_{k}$ from top to bottom, and still working in $Y_{k}$, we strengthen each $p^{2,|Z|}_{\rho, s}$ to a condition $p_{\rho, s}$ which decides each value $\delta_{p_{\rho, s}(\alpha_{i})}$ ($i \in u_{k}$), while preserving item (\ref{rhoscondtwo}) from earlier in this proof and thus condition (\ref{itemeight}) from Definition \ref{condtreedef}. For each $i \in u_{k}$ we can accomplish this in $|Z|$ many steps via an argument similar to the ones used in stages 1 and 2. We now choose conditions $p^{3,i,\ell}_{\rho, s}$ for $(\rho, s) \in Z$, $i \in u_{k}$, $\ell \leq |Z|$ such that
\begin{enumerate}
  \setcounter{enumi}{4}
  \item for all $(\rho, s) \in Z$, 
  \begin{enumerate}
     \item $p^{3,\max(u_{k}),0}_{\rho,s} = p^{2,|Z|}_{\rho, s}$; 
     \item for all $i \in u_{k}$ and $\ell < |Z|$, $p^{3,i,\ell+1}_{\rho, s} \leq p^{3,i,\ell}_{\rho,s}$;
     \item\label{rhoscondfivec} for all $i \in u_{k}$, $\ell \leq |Z|$ and $\alpha \in \alpha_{*} \setminus \alpha_{i}$, $p^{3,i,\ell+1}_{\rho,s}(\alpha) = p^{3,i,\ell}_{\rho,s}(\alpha)$;
     \item when $i' < i$ are successive members of $u_{k}$, $p^{3,i,|Z|}_{\rho, s} = p^{3,i',0}_{\rho, s}$;
     \item for all $i \in u_{k}$ and $\ell < |Z|$, $p^{3,i,\ell+1}_{\rho, s}\restrict \alpha_{i}$ decides the value of 
     $\delta_{p^{3,i,\ell+1}_{\rho, s}(\alpha_{i})}$;
  \end{enumerate}
  \item\label{coherence3} for all $i \in u_{k}$ and $\ell \leq |Z|$, and all distinct $(\rho, s), (\rho', s') \in Z$, the least $\alpha$ such that
\[p^{3,i,\ell}_{\rho, s}(\alpha) \neq p^{3,i,\ell}_{\rho', s'}(\alpha)\] is the same as the least $\alpha$ such that
\[p^{2,0}_{\rho, s}(\alpha) \neq p^{2,0}_{\rho', s'}(\alpha).\]
\end{enumerate}
In the substage of stage 3 corresponding to $i \in u_{k}$ and $\ell < |Z|$, strengthen $p^{3,i,\ell}_{\rho_{\ell}, s_{\ell}}$ (possibly trivially, and only in coordinates below $\alpha_{i}$) to a condition $p^{3,\ell+1}_{\rho_{\ell},s_{\ell}}$ deciding the value of $\delta_{p^{3,i,\ell}_{\rho_{\ell}, s_{\ell}}(\alpha_{i})}$.
For each of the other conditions $p^{3,i,\ell}_{\rho, s}$, letting $\alpha$ be minimal such that \[p^{3,i,\ell}_{\rho_{\ell}, s_{\ell}}(\alpha) \neq p^{3,i,\ell}_{\rho, s}(\alpha)\]
let $p^{3,i,\ell + 1}_{\rho, s}$ be the concatenation of $p^{3,i,\ell + 1}_{\rho_{\ell}, s_{\ell}} \restrict \alpha$ with $p^{3,i,\ell}_{\rho, s} \restrict [\alpha, \alpha_{*})$.

Finally, for each $(\rho, s)\in Z$ let $p_{\rho^{\frown}\langle t_{\rho, s} \rangle} = p_{\rho, s}$ be $p^{3,\min(u_{k}),|Z|}_{\rho,s}$. 
\end{proof}







We will need to consider restrictions of our condition trees to initial segments of our iterations. Given a tree $T$ in $\cT$, we let $[T]$ denote the set of infinite branches through $T$.

\begin{df}\label{projdefs} Suppose that $(T, \langle p_{\rho} : \rho \in T\rangle)$ is a condition tree relative to some suitable data sequence
\[S = \langle \bar{C}, F, \bq, p, X, \langle Y_{k} : k < \omega \rangle, \langle i_{n} : n < \omega \rangle, \bar{D} \rangle,\]
$\gamma = X \cap \omega_{1}$, $\alpha_{*}$ is the length of $\bq$, and that $\langle \alpha_{i} : i < i_{*} \rangle$ enumerates $X \cap \alpha_{*}$ in increasing order.
For each $k \in \omega$, let $u_{k} = \{ i_{n} : n < k \}$. Let $\alpha_{i_{*}}$ denote $\alpha_{*}$.
\begin{itemize}
\item For each $j \leq i_{*}$, let
\begin{itemize}
\item for each $k \in \omega$, $u_{k}(j)$ be $u_{k} \cap j$;
\item for each $\rho \in T$, $\rho^{[j]}$ be \[\langle s^{\rho}_{\ell} \restrict u_{\ell}(j) : \ell < |\rho|\rangle,\]
let $s^{\rho^{[j]}}_{\ell}$ (for $\ell < |\rho^{[j]}|$) denote $s^{\rho}_{\ell} \restrict u_{\ell}(j)$,
and $p^{[j]}_{\rho}$ be $p_{\rho} \restrict \alpha_{j}$;
\item $T^{[j]}$ be the tree of sequences $\{\rho^{[j]}: \rho \in T\}$.
\end{itemize}
\item For all $i \leq j < i_{*}$ and each $\rho \in T^{[j]}$,  let $\rho^{[i]}$ be $\langle \rho(\ell) \restrict u_{\ell}(i) : \ell < |\rho|\rangle$.
\item For all $i \leq j \leq i_{*}$, let
\begin{itemize}
\item $\proj_{j,i} : T^{[j]} \to T^{[i]}$ be the function defined by setting
\[\proj_{j,i}(\rho) = \rho^{[i]};\]

\item for each $x \in [T^{[i]}]$, $T^{[j]}(x)$ be $\{ \rho \in T^{[j]} :  \rho^{[i]} \trianglelefteq x\}$;
\item for each $x \in [T^{[j]}]$, $x^{[i]}$ be the union of $\{ (x \restrict k)^{[i]} : k \in \omega \}$.
\end{itemize}
\end{itemize}
\end{df}

We record some observations on these definitions.

\begin{thrm}\label{computationthrm} Suppose that $(T, \langle p_{\rho} : \rho \in T\rangle)$ is a condition tree relative to some suitable data sequence
\[S = \langle \bar{C}, F, \bq, p, X, \langle Y_{k} : k < \omega \rangle, \langle i_{n} : n < \omega \rangle, \bar{D} \rangle,\]
$\gamma = X \cap \omega_{1}$, $\alpha_{*}$ is the length of $\bq$, and that $\langle \alpha_{i} : i < i_{*} \rangle$ enumerates $X \cap \alpha_{*}$ in increasing order.
For each $\ell \in \omega$, let $u_{\ell} = \{ i_{n} : n < \ell, \alpha_{i_{n}} \in Y_{\ell} \}$. Let $\alpha_{i_{*}}$ denote $\alpha$.
\begin{enumerate}
  \item For all $\rho \in T$ and $\ell \in \omega$, $\rho^{[i_{*}]} = \rho$ and $u_{\ell}(i_{*}) = u_{\ell}$.
  \item\label{remtwo} If $j \leq i_{*}$ and $\rho \in T^{[j]}$ has length $k + 1$, then $\rho$ has exactly $2^{|u_{k}(j)|(m_{k+1} - m_{k})}$ many immediate successors in $T^{[j]}$.
  \item\label{remthree} For all $i \leq j \leq i_{*}$ and $x \in [T^{[i]}]$, if $\rho \in T^{[j]}(x)$ has length $k+1$, then $\rho$ has exactly $2^{|u_{k}(j) \setminus u_{k}(i)|(m_{k+1} - m_{k})}$ many successors in $T^{[j]}(x)$.
  \item\label{remfour} For all $i \leq j \leq i_{*}$, $k \in \omega \setminus \{0\}$, $\rho \in T^{[i]}$ of length $k+1$ and $\nu \in T^{[j]}$ of length $k$ with $\nu^{[i]} = \rho \restrict k$, the set \[\{ \tau \in T^{[j]} : \tau \restrict k = \nu, \, \tau^{[i]} = \rho\}\] has size
      $2^{|u_{k-1}(j)\setminus u_{k-1}(i)|(m_{k} - m_{k-1})}$.
  \item\label{remfive} For all $c \leq i \leq j \leq i_{*}$, $x \in [T^{[c]}]$,  $k \in \omega \setminus\{0\}$, $\rho \in T^{[i]}(x)$ of length $k + 1$ and $\nu \in T^{[j]}(x)$ of length $k$ with $\nu^{[i]} = \rho \restrict k$, the set \[\{ \tau \in T^{[j]}(x) : \tau \restrict k = \nu, \, \tau^{[i]} = \rho\}\] has size  $2^{|u_{k-1}(j)\setminus u_{k-1}(i)|(m_{k} - m_{k-1})}$.


\end{enumerate}
\end{thrm}

\begin{proof}
  The first conclusion follows immediately from the definitions. For part (\ref{remtwo}), note that by parts (\ref{itemtwob}), (\ref{itemthree}) and (\ref{itemthreepfive}) of Definition \ref{condtreedef}, the immediate successors of a $\rho \in T^{[j]}$ of length $k + 1$ correspond to the functions from $u_{k}(j)$ to $2^{m_{k+1}}$ which pointwise extend $s^{\rho}_{k}$, of which there are $2^{|u_{k}(j)|(m_{k+1} - m_{k})}$ many.
  For part (\ref{remtwo}), the same parts of Definition \ref{condtreedef} imply that the immediate successors of a $\rho \in T^{[j]}(x)$ of length $k + 1$ correspond to the functions from $u_{k}(j)\setminus u_{k}(i)$ to $2^{m_{k+1}}$ which pointwise extend $s^{\rho}_{k}\restrict (u_{k}(j) \setminus u_{k}(i))$, of which there are $2^{|u_{k}(j)\setminus u_{k}(i)|(m_{k+1} - m_{k})}$ many.

  For part (\ref{remfour}), the immediate successors of $\nu$ correspond to the functions $s \colon u_{k-1}(j) \to 2^{m_{k}}$ which pointwise extend $s^{\nu}_{k-1}$, and the successors $\tau$ of $\nu$ with $\tau^{[i]} = \rho$ correspond to those $s$ which agree with $s^{\rho}_{k}$ on $u_{k-1}(i)$. The number of such functions is $2^{|u_{k-1}(j)\setminus u_{k-1}(i)|(m_{k} - m_{k-1})}$. Part (\ref{remfive}) is similar.
\end{proof}

\begin{remark} Using the objects introduced in the statement of Theorem \ref{computationthrm}, fix for each pair $(i,j)$ with $i < j \leq i_{*}$ a tree $T_{i,j}$ of finite sequences with unique nodes of lengths $0$ and $1$ and each node of each length $k > 0$ having exactly $2^{|u_{k-1}(j) - u_{k-1}(i)|(m_{k}-m_{k-1})}$ many immediate successors. The computations in Theorem \ref{computationthrm} show that each tree $T^{[j]}$ is isomorphic to the tree of pairs of sequences $(\rho, \nu)$ of the same length with $\rho \in T^{[i]}$ and $\nu \in T_{i,j}$. Similarly, each tree $T^{[j]}(x)$ for some $x \in [T^{[c]}]$ for some $c \leq i$ is isomorphic to the tree of pairs of sequences $(\rho, \nu)$ of the same length with $\rho \in T^{[i]}(x)$ and $\nu \in T_{i,j}$. Under these identifications, the function $\proj_{j,i}$ from Definition \ref{projdefs} corresponds to the usual first-coordinate projection function.
These facts let us apply Fubini's theorem and the Kuratowksi-Ulam theorem to pairs $[T^{[i]}], [T^{[j]}]$ (with the latter viewed as $[T^{[i]}] \times [T_{i,j}]$) and pairs $[T^{[i]}(x)], [T^{[j]}(x)]$ (with the latter viewed as $[T^{[i]}(x)] \times [T_{i,j}]$), which we do in the proof of Lemma \ref{finallemma}.
\end{remark}

The following observation will be used at the end of the proof of Lemma \ref{finallemma}.

\begin{remark}\label{tclosedrem}
Suppose that $(T, \langle p_{\rho} :\rho \in T \rangle)$ is a condition tree for a suitable data sequence $S = \langle \bar{C}, F, \bq, p, X, \bar{Y}, \bar{\i}, \bar{D} \rangle$, and that $i_{*}$ is the ordertype of $X \cap \alpha_{*}$, where $\alpha_{*}$ is the ordertype of the iteration $\bq$. Let $j \leq i_{*}$ and let $K$ be a closed subset of $[T^{[j]}]$. Then for all $x \in [T^{[j]}]$, $x \in K$ if and only if, for cofinally many $c < j$, $x^{[c]} \in \proj_{j,c}[K]$.
\end{remark}

\section{Preserving pathology}\label{pathsec}



In this section we add a condition on our dense partial function $F$, and recall some important properties of the null and meager ideals. We say that an ideal $I$ on a Polish space is \emph{Borel} if it is generated by Borel sets (i.e., every member of the ideal is contained in a Borel member of the ideal).  The meager and null ideals are Borel in this sense. Given an ideal $I$ on a set $X$, we let $I^{+}$ denote $\cP(X) \setminus I$, and we say that a subset of $X$ is $I$-\emph{Borel-large} if it intersects every Borel set in $I^{+}$.



\begin{df}
  Given a topological space $X$, an ideal $I$ on $X$, and a set $S$ with at least two elements, we say that a partial function $F \colon X \to S$ is $I$-\emph{pathological} if for
  every $s \in S$, $F^{-1}[\{s\}]$ is $I$-Borel-large.
\end{df}


When $I$ is the ideal of countable sets, we say that $F$ is \emph{totally pathological}.
We say that $F$ is \emph{Lebesgue-pathological} when $I$ is the ideal of Lebesgue null sets, and
\emph{category-pathological} when $I$ is the ideal of meager sets.
Note that total pathology implies both Lebesgue and category pathology.

\begin{remark}A totally pathological function is dense in the sense of Definition \ref{denfundef}, and the existence of a totally pathological function implies
the existence of sets of reals without the perfect set property.
\end{remark}

A standard construction, using a wellordering of the continuum and the fact that uncountable Borel sets have cardinality continuum, shows that $\ZFC$ implies the existence of (total) totally pathological functions on $2^{\omega}$.
Moreover, if $a$ is a subset of $\omega$, then the same construction, using a $\Sigma^{1}_{2}(a)$ wellordering of $(2^{\omega})^{\bL[a]}$,
shows that there is in $\bL[a]$ a total totally pathological function $F \colon 2^{\omega} \to 2$ which is $\Sigma^{1}_{2}$ in $a$.

In our proof in Section \ref{itsec} we iterate forcings of the form $Q_{\bar{C}, F, g}$ using a fixed function $F$ which is totally pathological and $\Sigma^{1}_{2}(a)$ in the ground model $\bL[a]$. Since our iterations will add reals, the function $F$ will not remain totally pathological in the corresponding forcing extensions.  For the proof of our main theorem, we need to know that $F$ remains $I$-pathological throughout the iteration, where $I$ is either the null ideal or the meager ideal.

\begin{df}\label{prespropdef}  A Borel ideal $I$ has the \emph{preservation property} if every $I$-Borel-large set remains so after any countable support iteration of partial orders of the form $Q_{\bar{C}, F, g}$.
\end{df}

\begin{remark}\label{pathpresrem} The ideals that we consider in this paper are ideals on the set of infinite branches through finitely branching trees of height $\omega$.
Theorem 5.2 of \cite{KellSh} implies that the versions of the meager ideal on these spaces have the preservation property (for a much wider class of partial orders than the ones considered here). Theorem 6.3 and Lemma 6.4 of \cite{KellSh} show the same thing for the ideals of Lebesgue null sets.
\end{remark}

The proof of Theorem \ref{mainthrm} will also use the fact that the meager and null ideals are regular, in the sense that every Borel set not in either ideal contains a closed set which is not in the same ideal.

\section{Measure and Category}\label{itsec}

In this section we prove our main theorem.

\begin{thrm}\label{mainthrm} Suppose that $a$ is a subset of $\omega$ such that $\bV = \bL[a]$, and let $\bar{C}_{a}$ be the canonical ladder system relative to $a$. Let $F$ be a totally pathological dense function whose graph is $\uTSigma^{1}_{2}$ in $\bL[a]$.
Suppose that \[\langle \bbP_{\alpha}, \dot{Q}_{\beta}, \dot{g}_{\beta}, \dot{h}_{\beta} : \alpha \leq \omega_{2}^{\bV}, \beta < \omega_{2}^{\bV} \rangle\]
is a fully bookkeeping element of $\bQ_{\bar{C}_{a}, F}$.
Then $\bbP_{\omega_{2}^{\bV}}$ forces that every subset of $\breals$ which is either universally measurable or universally categorical is $\uTDelta^{1}_{2}$.
\end{thrm}

The proof of Theorem \ref{mainthrm} runs from here through the end of this section. We will concentrate on the universally measurable case, and note the changes that need to be made for universally categorical sets.


By Remark \ref{reducerem}, it suffices to prove that each tail of any iteration as in the statement of Theorem \ref{mainthrm} is $\cA$-representing, where $\cA$ is either the collection of universally measurable subsets of $\breals$ or the collection of universally categorical sets.
Since the ideals of Lebesgue-null sets and meager sets have the preservation property (as in Remark \ref{pathpresrem}), in each intermediate model of such a forcing iteration, $F$ is Lebesgue-pathological and category-pathological (but no longer totally pathological once the iteration has added new subsets of $\omega$).
It remains then to prove the following.

\begin{lem}\label{representinglem}
  Suppose that
  \begin{itemize}
  \item $\bar{C}$ is a ladder system on $\omega_{1}$,
  \item $F \colon 2^{\omega} \to 2$ is a dense function which is Lebesgue-pathological and category-pathological,
  \item $\bq = \langle \bbP_{\alpha}, \dot{Q}_{\beta}, \dot{g}_{\beta}, \dot{h}_{\beta} : \alpha \leq \alpha_{*}, \beta < \alpha_{*} \rangle$ is in $\bQ_{\bar{C}, F}$ and
  \item $\bar{\bbP}$ is forcing iteration $\langle \bbP_{\alpha}, \dot{Q}_{\beta} : \alpha \leq \alpha_{*}, \beta < \alpha_{*}\rangle$.
  \end{itemize}
  Then $\bbP_{\alpha_{*}}$ is $\cA$-representing, where $\cA$ is either the collection of universally measurable subsets of $\breals$ or the collection of universally categorical sets.
\end{lem}

Fix the objects introduced in the statement of Lemma \ref{representinglem} from here through the end of proof of Lemma \ref{finallemma}, which will complete the proof of Lemma \ref{representinglem} and thus Theorem \ref{mainthrm}.
We have to show that if
  \begin{itemize}
    \item $p$ is a condition in $\bbP_{\alpha_{*}}$,
    \item $A \subseteq \breals$ is either universally measurable or universally categorical,
    \item $\tau$ is a $\bbP_{\alpha_{*}}$-name for an element of $\breals$
  \end{itemize}
  then there exist a condition $p' \leq p$ and a Borel set $B \subseteq \breals$ such that
  \begin{itemize}
    \item $B$ is either contained in or disjoint from $A$ and
    \item $p' \forces \tau \in \check{B}$.
  \end{itemize}
  In fact the set we find will be a continuous image of a Borel set; as every analytic set is a union of $\aleph_{1}$ many Borel sets (in an absolute way, see page 201 of \cite{Ke}), this suffices.
  To find our desired condition and analytic set, fix such $p$, $A$ and $\tau$, and fix in addition $X$, $\bar{Y}$, $\bar{\i}$ and $\bar{D}$ (with $\tau \in X$)  such that
  $S = \langle \bar{C}, F, \bq, p, X, \bar{Y}, \bar{\i}, \bar{D} \rangle$ is a suitable data sequence. By Lemma \ref{fusionlemma}, we may fix in addition
  a condition tree $(T, \langle p_{\rho} : \rho \in T\rangle)$ relative to $S$.
  There is then a continuous function $f \colon [T] \to \breals$ such that each value $f(x)$ is the realization of $\tau$ by
  $\{ p_{x \restrict n} : n \in \omega \}$, since any such set meets all the dense open subsets of $\bbP_{\alpha_{*}}$ in $X$. Since $A$ is universally measurable (or universally categorical), there exist then Borel sets $B$ and $N$ contained in $[T]$ such that $N$ is Lebesgue null (or meager) and $f^{-1}[A] \bigtriangleup B \subseteq N$. It suffices then to find a condition $p' \leq p$ forcing the existence (in the forcing extension) of an $x \in [T] \setminus N$ (reinterpreted) such that $\{ p_{x \restrict n} : n \in \omega\}$ is a subset of the generic filter. We will in fact start with a closed non-null (or nonmeager) set $K \subseteq [T]$ disjoint from $N$ and find a condition $p' \leq p$ forcing the existence of an $x \in K$ (reinterpreted)  such that $\{ p_{x \restrict n} : n \in \omega\}$ is a subset of the generic filter.
An instance of such a condition will be a Lebesgue-$(i_{*}, \check{K})$-solution (respectively, a category-$(i_{*}, \check{K})$-solution) relative to $S$ and $(T, \langle p_{\rho} : \rho \in T\rangle)$, as in Definition \ref{soldef} below.

Given $j \leq i_{*}$, we let $\dot{\sigma}_{j}$ be a $\bbP_{j}$-name for the set of $\rho \in T^{[j]}$ for which $p_{\rho} \in G_{j}$, where $G_{j}$ denotes the restriction of the generic filter $G$ to $\bbP_{\alpha_{j}}$. In the desired case where this set is infinite, we identify it with the corresponding element of $[T^{[j]}]$.


%
%

\begin{df}\label{soldef}
 Suppose that $(T, \langle p_{\rho} : \rho \in T\rangle)$ is a condition tree relative to some suitable data sequence
\[S = \langle \bar{C}, F, \bq, p, X, \langle Y_{k} : k \in \omega \rangle, \langle i_{n} : n < \omega \rangle, \bar{D} \rangle.\]
Let
\begin{itemize}
\item $\bq$ be $\langle \bbP_{\alpha}, \dot{Q}_{\beta}, \dot{g}_{\beta}, \dot{h}_{\beta} : \alpha \leq \alpha_{*}, \beta < \alpha_{*} \rangle$;
\item $\langle \alpha_{i}  : i < i_{*} \rangle$ enumerate $X \cap \alpha_{*}$ in increasing order;
\item $\alpha_{i_{*}}$ denote $\alpha_{*}$;
\item $j$ be an element of $i_{*}+1$;
\item $\dot{K}$ be a $\bbP_{j}$-name for a subset of $[T]$;
\item $q$ be a condition in $\bbP_{\alpha_{j}}$;
\end{itemize}
We say that $q$ is a Lebesgue-$(j,\dot{K})$-solution if $q$ forces that
\begin{itemize}
\item $\dot{\sigma}_{j} \in [T^{[j]}]$;
\item if $j < i_{*}$ then $\dot{K} \cap [T(\dot{\sigma}_{j})]$ is non-null in $[T(\dot{\sigma}_{j})]$;
\item if $j = i_{*}$ then $\dot{\sigma}_{j} \in \dot{K}$.
\end{itemize}
We say that $q$ is a category-$(j,\dot{K})$-solution if $q$ forces that
\begin{itemize}
\item $\dot{\sigma}_{j} \in [T^{[j]}]$;
\item if $j < i_{*}$ then $\dot{K} \cap [T(\dot{\sigma}_{j})]$ is nonmeager in $[T(\dot{\sigma}_{j})]$;
\item if $j = i_{*}$ then $\dot{\sigma}_{j} \in \dot{K}$.
\end{itemize}
\end{df}

The statement of the following lemma uses the objects introduced in this section.
The case where $i=0$, $j = i_{*}$ and $\dot{K} = \check{K}$ for some non-null closed $K$ disjoint from $N$ as above proves Lemma \ref{representinglem} and thereby completes the proof of Theorem \ref{mainthrm}.

\begin{lem}\label{finallemma}
Suppose that $i < j \leq i_{*}$, $\dot{K}$ is a $\bbP_{i}$-name for a closed subset of $[T]$ and $q$ is a Lebesgue-$(i, \dot{K})$-solution (respectively, a
category-$(i, \dot{K})$-solution).
Then there is a Lebesgue-$(j,\dot{K})$-solution (category-$(j,\dot{K})$-solution) $q' \in \bbP_{\alpha_{j}}$ such that $q' \restrict \alpha_{i} = q$.
\end{lem}


\begin{proof}
We prove this by induction on $j$ for all $i$ and $\dot{K}$ simultaneously.
We first note that the case $j = j' + 1$ reduces to the case where $j = i + 1$ by the induction hypothesis.


For the case where $j = i + 1$, suppose that $G_{i} \subseteq \bbP_{\alpha_{i}}$ is a generic filter with $q \in G_{i}$.
Then $G_{i} \cap X$ is $X$-generic. Let $e$ be $\dot{g}_{\alpha_{i},G_{i}}(X \cap \omega_{1})$, let  $y = \dot{\sigma}_{i,G_{i}}$ and let
$\hat{y}$  be the function from $i$ to $2^{\omega}$ induced by $y$, so that, for all $c < i$, $\hat{y}(c) = \bigcup\{s^{\rho}_{k}(c) : \rho \in y,\, k < |\rho|,\, c \in u_{k}\}$.
Applying Fubini's theorem (or the Kuratowksi-Ulam theorem) and the regularity of the null ideal (or meager ideal), let $K'$ be a closed non-null (nonmeager) subset of $[T^{[j]}(y)]$ with the property that for each $z \in K'$ the set $\{ x \in K_{G_{i}} : \proj_{i_{*},j}(x) = z\}$ is non-null (nonmeager).


Let $k \in \omega$ be minimal with $i\in u_{k}$. Then $k > 0$. Since $i$ is not in $u_{k-1}$, the tree $T^{[j]}(y)$ has a unique node $\rho_{k}$ of length $k$, where for each $\ell < k$ and each $c \in u_{\ell}(j)$, $c$ is in $u_{\ell}(i)$ and
$s^{\rho_{k}}_{\ell}(c) = \hat{y}(c)\restrict m_{\ell}$. Moreover, for all $\rho$, $\rho'$ of length $k+1$ in $T(y)$ (formally defined as $T^{[i_{*}]}(y)$), $s^{\rho}_{k} \restrict u_{k-1}(j) = s^{\rho'}_{k} \restrict u_{k-1}(j)$ (as just established), so, by condition
(\ref{itemthreepfive}) of Definition \ref{condtreedef}, $s^{\rho'}_{k} \restrict u_{k}(j) = s^{\rho'}_{k} \restrict u_{k}(j)$.
It follows that $T^{[j]}(y)$ has a unique node $\rho_{k+1}$ of length $k+ 1$.
The condition $p_{\rho_{k+1}} \restrict \alpha_{j}$ forces
\[\langle \dot{h}_{{\alpha}_{i}}(C_{X \cap \omega_{1}}(n)) : n < m_{k}\rangle\] to be $s^{\rho_{k+1}}_{k}(i)$, which is in $2^{m_{k}}$.
For each $\ell \geq k+1$, each node $\rho$ of $T^{[j]}(y)$ of length $\ell$ has $2^{m_{\ell} - m_{\ell-1}}$ many immediate successors, corresponding to the possible extensions of $s^{\rho}_{\ell-1}(i)$ of length $m_{\ell}$.

Let $T_{*}$ be the tree of elements of $2^{\less\omega}$ compatible with the sequence $s^{\rho_{k+1}}_{k}(i)$.
Thus $[T_{*}]$ is a clopen subset of $2^{\omega}$, and equal to \[\{ \bigcup\{ s^{z \restrict (\ell + 1)}_{\ell}(i) : \ell \in \omega \setminus k\} : z \in [T^{[j]}(y)]\}.\]
By Remark \ref{pathpresrem}, $F$ is both Lebesgue-pathological and category-pathological in $V[G_{i}]$.
Since the null (meager) ideal on $[T_{*}]$ is
isomorphic to that on $[T^{[j]}(y)]$
there is a $z \in K'$ such that \[F\left(\bigcup_{\ell \in [k, \omega)}s^{z \restrict (\ell + 1)}_{\ell}(i)\right) = e.\]
We can then let $q'$ be $(q,\dot{r})$, where $\dot{r}$ is a $\bbP_{\alpha_{i}}$ name for $\bigcup \{ s^{z \restrict (\ell + 1)}_{\ell}(i) : \ell \in \omega \setminus k\}$, which is an $X[G_{i} \cap X]$-generic condition in $\dot{Q}_{\alpha_{i}, G_{i}}$.

For the case where $j$ is a limit ordinal, applying Fubini's theorem (or the Kuratowksi-Ulam theorem) again, let $\dot{K}_{0}$ be a $\bbP_{i}$-name for a closed subset of (the realization of) $\dot{K}$ such that
$q$ is a Lebesgue-$(i,\dot{K}_{0})$-solution (category-$(i,\dot{K}_{0})$-solution) and, $q$ forces in addition that for all $y \in \proj_{i_{*}, j}[\dot{K}_{0}]$ the set $\dot{K} \cap [T(y)]$
is non-null (nonmeager) (of course, this step is not necessary if $j = i_{*}$). This case now follows from the induction hypothesis applied successively to $\dot{K}_{0}$ and the members of any cofinal $\omega$-seqence in the interval $(i,j)$, along with Remark \ref{tclosedrem}.
\end{proof}

\section{Borel reinterpreations}

Finally, we show that the Borel reinterpretations of universally measurable sets under the forcing extensions considered here are again universally measurable.

\begin{thrm}\label{umeasreint}
Let $\bar{C}$ be a ladder system on $\omega_{1}$
and let $F \colon 2^{\omega} \to 2$ be a null-pathological dense partial function.
Let \[ \langle \bbP_{\alpha}, \dot{Q}_{\beta}, \dot{g}_{\beta}, \dot{h}_{\beta} : \alpha \leq \alpha_{*}, \beta < \alpha_{*} \rangle,\]
be an element of $\bQ_{\bar{C}, F}$. If $A \subseteq \breals$ is universally measurable, then the Borel reinterpretation of $A$ is
universally measurable in any forcing extension by $\bbP_{\alpha_{*}}$.
\end{thrm}

\begin{proof}
  It follows from Lemma \ref{representinglem} that the Borel reinterpretations of $A$ and $\breals \setminus A$ will be complements in any forcing extension by $\bbP_{\alpha_{*}}$. We have to show that if
  \begin{itemize}
    \item $p$ is a condition in $\bbP_{\alpha_{*}}$,
    \item $A \subseteq \breals$ is universally measurable and
    \item $\tau$ is a $\bbP_{\alpha_{*}}$-name for a Borel measure on $\breals$
  \end{itemize}
  then there exist a condition $p' \leq p$ and a $\bbP_{\alpha_{*}}$-name $\dot{B}$ for a Borel subset $\breals$ such that $p'$ forces the symmetric difference of $\dot{B}$ and the Borel reinterpretation of $A$ to be $\tau$-null.
  To do this, fix such $p$, $A$ and $\tau$, and fix in addition $X$, $\bar{Y}$, $\bar{\i}$ and $\bar{D}$ (with $\tau \in X$) such that
  $S = \langle \bar{C}, F, \bq, p, X, \bar{Y}, \bar{\i}, \bar{D} \rangle$ is a suitable data sequence. By Lemma \ref{fusionlemma}, we may fix in addition a condition tree \[(T, \langle p_{\rho} : \rho \in T\rangle)\] relative to $S$.
  There is then a continuous function $f$ from $[T]$ to the space of Borel measures on $\breals$ such that each value $f(x)$ is the realization of $\tau$ by
  $\{ p_{x \restrict n} : n \in \omega \}$. Let $\lambda$ be Lebesgue measure for $[T]$ and let $\nu$ be the measure on $\breals$ defined by setting $\nu(B)$ to be $\int_{[T]} f(x)(B)\,d\lambda$. Since $A$ is universally measurable,
  there exist Borel sets $B$ and $E$ contained in $\breals$ such that $E$ is $\nu$-null and $A \bigtriangleup B \subseteq E$.
  The same relationship will then hold for the (Borel) reinterpretations of $A$, $B$ and $E$ in any forcing extension by $\bbP_{\alpha_{*}}$.
  Let $K$ be a closed non-null set of $x \in [T]$ for which $E$ is $f(x)$-null.

  The statement that $E$ is $f(x)$-null for each $x \in K$ is $\Pi^{1}_{2}$ in codes for $f$, $E$ and $K$, and therefore absolute. One way to see this is to note that for any wellfounded model $M$ of a sufficiently large fragment of $\ZFC$ containing these codes and some $x \in K$, $M$ computes some values for the $f(x)$-measures of $E$ and its complement, and it must compute them correctly.

  We can then apply Lemma \ref{finallemma} for $K$ in the case $i=0$, $j = i_{*}$.
  The resulting condition $q'$ then forces that the realization of $\tau$ in the $\bbP_{\alpha_{*}}$-extension will be in the $f$-image of $K$. To see that the reinterpretations of $B$ and $E$ in this extension will witness that the Borel reinterpretation of $A$ is measurable relative to the realization of $\tau$, we need to see that the reinterpretation of $E$ is $f(x)$-null in this extension for all $x$ in the reinterpretation of $K$.
  \end{proof}

  Essentially the same argument shows that extensions by partial orders of the form $\bbP_{\alpha_{*}}$ are measured, in the sense of \cite{LZpime}.

\begin{remark} A set $A \subseteq \breals$ is universally null if $f^{-1}[A]$ is Lebesgue-null whenever $f \colon \breals \to \breals$ is a Borel function, and universally meager if $f^{-1}[A]$ is meager whenever $f \colon \breals \to \breals$ is a Borel function. Lemma \ref{representinglem} (along with the fact that universally null and universally meager sets cannot contain perfect sets) shows that in the forcing extensions we consider here, all universally null or universally meager sets have cardinality at most $\aleph_{1}$. It is an open question whether the classes of universally null and universally meager sets can coincide. This question is not answered by the arguments in this paper. A standard recursive construction (a version of which appears in \cite{LrNeSh}; see also \cite{BL}) shows that the Continuum Hypothesis implies the existence of nonmeager universally null sets. Since the iterations considered here preserve nonmeager sets (via Theorem 5.2 of \cite{KellSh} as cited in Remark \ref{pathpresrem}), Theorem \ref{umeasreint} implies that there exist such sets in the models produced in this paper. Since universally null sets cannot contain perfect sets, a nonmeager universally null set cannot have the property of Baire. This shows that there are universally measurable sets in the extensions produced in this paper which are not universally categorical.
\end{remark}

\begin{remark} When $a \subseteq \omega$, the inner model $\bL[a]$ contains a $\Delta^{1}_{2}(a)$ wellordering of its reals which remains $\Delta^{1}_{2}(a)$ in any outer model. Since the iterations considered here preserve non-null sets (via Theorem 6.3 and Lemma 6.4 of \cite{KellSh} as cited in Remark \ref{pathpresrem}), the reals of $\bL[a]$ form a non-null set in the models considered here. By \cite{Reclaw}, any set of reals wellordered by a universally measurable relation is universally null. It follows in the extensions produced in this paper, the canonical wellordering of the reals constructed in $\bL[a]$ is a $\Delta^{1}_{2}(a)$ set which is not universally measurable. The analogous results for category (\cite{KellSh} and \cite{Reclaw} again) show that this set is also not universally categorical.
\end{remark}

\end{document}